\documentclass[12pt,a4paper]{amsart}
\usepackage{url}
\usepackage{amsmath,amsthm,amsfonts,amssymb,latexsym}

\newtheorem{theorem}{Theorem}
\newtheorem{proposition}[theorem]{Proposition}
\newtheorem{lemma}[theorem]{Lemma}

\newtheorem{corollary}[theorem]{Corollary}

\theoremstyle{definition}

\newcommand{\U}{\mathcal U}
\newcommand{\w}{\omega}

\newcommand{\B}{\mathcal{B}}
\newcommand{\C}{\mathcal{C}}

\newcommand{\D}{\mathcal{D}}
\newcommand{\W}{\mathcal{W}}
\newcommand{\F}{\mathcal{F}}

\newcommand{\V}{\mathcal{V}}

\newcommand{\dd}{\mathfrak d}

\newcommand{\uhr}{\upharpoonright}

\newcommand{\la}{\langle}
\newcommand{\ra}{\rangle}

\newcommand{\nothing}[1]{}

\title[Non-CDH filters in ZFC]{Countable dense homogeneous filters and the
Menger covering property}

\author{Du\v{s}an Repov\v{s}, Lyubomyr Zdomskyy, and Shuguo Zhang}

\address{Faculty of Education, and Faculty of Mathematics and Physics,
University of Ljubljana, P. O. Box 2964, Ljubljana, Slovenia 1001.}
\email{dusan.repovs@guest.arnes.si}
\urladdr{http://www.fmf.uni-lj.si/\~{}repovs/index.htm}

\address{Kurt G\"odel Research Center for Mathematical Logic,
University of Vienna, W\"ahringer Stra\ss e 25, A-1090 Wien,
Austria.}
\email{lzdomsky@gmail.com}
\urladdr{http://www.logic.univie.ac.at/\~{}lzdomsky/}

\address{College of Mathematics, Sichuan University, Chengdu, Sichuan 610064, China.}
\email{zhangsg@scu.edu.cn}

\subjclass[2010]{Primary: 54D20. Secondary:  54D80, 22A05.}
\keywords{CDH space, Menger space, Hurewicz space, $P$-filter, $\w$-cover,
groupable cover.}

\begin{document}

\begin{abstract}
In this note we present a ZFC construction of a non-meager filter which fails to be
 countable dense homogeneous. This  answers a question of
 Hern\'andez-Guti\'errez and Hru\v{s}\'ak. The method of the proof also allows
us to obtain   a metrizable
 Baire topological group   which is strongly
locally homogeneous but not countable dense homogeneous.

\end{abstract}

\maketitle

\section{Introduction}

A topological space $X$ has property \emph{CDH}  (abbreviated from
\emph{countable dense homogeneous}) if for arbitrary countable dense
subsets $D_0, D_1$ of $X$ there exists a
 homeomorphism $\phi:X\to X$ such that $\phi[D_0]=D_1$. The study of CDH
 filters on $\w$ was initiated
in \cite{MedMil12} where the property CDH was used to find concrete
examples of non-homeomorphic filters and ultrafilters, considered with the
topology inherited from $\mathcal P(\w)$.

  The following theorem is the main result of this note.
Let us stress that we do not use any additional set-theoretic
assumptions in its proof.

\begin{theorem} \label{main}
 There exists a non-meager non-CDH filter.
\end{theorem}

There are several constructions of non-CDH spaces  by transfinite induction using some
enumeration of all potential autohomeomorphisms, by adding points to the space under construction
 in such a way that
these homeomorphisms are ruled out one by one, see, e.g.,
\cite{MedMil12} and references therein. However, this method often
requires some equalities between cardinal characteristics and hence
does not seem to lead to a construction of non-meager non-CDH
filters outright in ZFC.

  Instead of ruling out all
 potential  autohomeomorphisms sending some countable dense subset $D_0$
  onto some other countable dense
subset $D_1$ one by one, we shall do this at once. Our idea is
rather straightforward: if a space $X$ admits two countable dense
subsets $D_0, D_1$ such that $X\setminus D_0$ is not homeomorphic to
$X\setminus D_1$, then there obviously is no autohomeomorphisms of
$X$  mapping $D_0$ onto $D_1$. We shall prove Theorem~\ref{main} by
constructing a non-meager filter $\F$ on $\w$ and two countable
dense subsets $\D_0,\D_1$ of $\F$ such that $\F\setminus\D_1$ has
the Menger property whereas $\F\setminus\D_0$ does not, see the next
section for its definition.

\section{Covering properties of Menger and Hurewicz}

We recall from
\cite{COC1}
that a topological space $X$ has
\begin{itemize}
\item \emph{the Menger (covering) property}, if   for every sequence $\la \U_n:n\in\omega\ra$
of open covers of $X$ there exists a sequence $\la \V_n : n\in\omega\ra$ such that
each $\V_n$ is a finite subfamily of $\U_n$ and the collection $\{\cup \V_n:n\in\omega\}$
is a cover of $X$;
\item \emph{the Hurewicz (covering) property}, if   for every sequence $\la \U_n:n\in\omega\ra$
of open covers of $X$ there exists a sequence $\la \V_n : n\in\omega\ra$ such that
each $\V_n$ is a finite subfamily of $\U_n$ and the collection $\{\cup \V_n:n\in\omega\}$
is a $\gamma$-cover of $X$
(a family $\U$ of subsets of a space $X$ is called a $\gamma$-cover of $X$ if
every $x\in X$ belongs to all but finitely many $U\in \U$).
\end{itemize}
These properties were introduced by Hurewicz in \cite{Hur25} and \cite{Hur27}, respectively.
It is clear that every $\sigma$-compact space has the Hurewicz property,
and the Hurewicz property implies the Menger one. It is known
\cite{ChaPol02, COC2, TsaZdo08} that none of these implications
can be reversed. The simplest example of a metrizable
space without the Menger property is the Baire space $\w^\w$.
Indeed $\la\U_n:n\in\w\ra$, where $\U_n=\{\{x\in\w^\w:x(n)=k\}:k\in\w\}$,
is a sequence of open covers of $\w^\w$ witnessing the failure of the
Menger property.

In the  proof of Theorem~\ref{main} we shall use without  mentioning
several basic facts about
these properties summarized in the following proposition.  Most likely these can be found
somewhere in the literature. However, we did not try to locate them or to present
their proofs as we believe
that the  proof of any of them should not take the reader more than a couple of minutes.

\begin{proposition} \label{easy_useful}
\begin{itemize}
 \item[$(i)$]  If a topological space $X$ has the Menger (Hurewicz) property and
 $Y$ is a continuous image of
$X$, then $Y$ is Menger (Hurewicz);
 \item[$(ii)$] If a topological space $X$ has the Menger (Hurewicz) property and
 $Y$ is a closed subspace of
$X$, then $Y$ is Menger (Hurewicz);
 \item[$(iii)$]   If a topological space $X$ has the Menger (Hurewicz) property and
 $Y$ is compact, then $X\times Y$ is Menger (Hurewicz);
 \item[$(iv)$] If $\{Y_i:i\in\w\}$ is a collection of Menger (Hurewicz) subspaces of
a space $X$, then $\bigcup_{i\in\w}Y_i$ is Menger (Hurewicz);
 \item[$(v)$]   If a topological space $X$ has the Menger (Hurewicz) property and
 $Y$ is an $F_\sigma$-subspace of
$X$, then $Y$ is Menger (Hurewicz);
\item[$(vi)$]   If a topological space $X$ has the Menger (Hurewicz) property and
 $Y$ is $\sigma$-compact, then $X\times Y$ is Menger (Hurewicz).
\end{itemize}
\end{proposition}

\begin{corollary}\label{easy1}
Let $X$ be a metrizable space with the Menger property. Then $X\setminus A$ has the Menger
property for all finite subsets $A$.
\end{corollary}
\begin{proof}
$X\setminus A$ is an  $F_\sigma$-subspace of $X$.
\end{proof}

By a filter on a countable set $C$ we mean a \emph{free filter},
i.e., a filter  containing all cofinite subsets of $C$. A family
$\B\subset\mathcal P(C)$ is said to be  \emph{centered} if
for any finite $\B'\subset\B$ the intersection
$\bigcap \B'$ is infinite. Any centered family generates
a filter in a natural way.

\begin{corollary} \label{fil1}
Suppose that a filter $\F$ on a countable set $C$ is generated
by a centered family $\B$ all of whose finite powers have the
Menger property when $\B$ is considered  with the subspace topology
inherited from $\mathcal P(C)$.    Then $\F$ is Menger.
\end{corollary}
\begin{proof}
Consider the map $\phi_n:\B^n\times\mathcal P(C)\times [C]^{<\w},$
$\phi:\la B_0,\ldots,B_{n-1};X;x\ra\mapsto (\bigcap_{i\in n}B_i\setminus x)\bigcup X$.
It is clear that each $\phi_n$ is continuous and
 $\F=\bigcup_{n\in\w}\phi_n[\B^n\times\mathcal P(C)\times [C]^{<\w}]$.
It suffices to use Proposition~\ref{easy_useful} several times.
\end{proof}

\section{Proof of Theorem~\ref{main}}

We shall  divide the proof into two lemmas.

\begin{lemma} \label{easy}
Let $\F$ be a filter on $\w$ containing co-infinite sets. Then there exists a countable
dense subset $\D$ of $\F$ such that $\F\setminus \D$ does not have the Menger property.
\end{lemma}
\begin{proof}
Let us find $F\in\F$  such that $|\w\setminus F|=\w$ and consider
the subspace $\C=\{F'\in\F:F\subset F'\}$ of $\F$. Let $\D',\D''$ be
countable dense subspaces of $\C$ and $\F\setminus\C$, respectively.
Notice that $\C$ is a copy of the Cantor set being compact
zero-dimensional space without isolated points.
 Therefore $\C\setminus \D'$ is homeomorphic to $\w^\w$. Thus  $\F\setminus(\D'\cup\D'')$
has a closed subspace homeomorphic to $\w^\w$ (namely $\C\setminus
\D'$)  and hence does not have the Menger property.
\end{proof}

  A collection $\U$ of subsets of $X$ is called an \emph{$\w$-cover}
of $X$ if $X\not\in\U$ and for every $A\in [X]^{<\w}$ there exists
$U\in\U$ such that $A\subset U$. A collection $\U$ is a
\emph{groupable $\w$-cover of $X$} if there is a partition
$\{\U_n:n\in\w\}$ of $\U$ into pairwise disjoint finite sets such
that for each finite subset $A$ of $X$ and for all but finitely many
$n$, there exists $U\in\U_n$ such that $A\subset\U$.

The existence of spaces $X$ such as in the following lemma was first established in
 \cite{ChaPol02}, see also  \cite[Corollary~6.4]{TsaZdo08}.

\begin{lemma} \label{easy2}
Let $X$ be a dense subspace of $\w^\w$ such that all finite powers
of $X$ have the Menger property but $X$ fails to have the Hurewicz
property. Then there exists a clopen $\w$-cover $\V$ of $X$ which
fails to be a groupable $\w$-cover of $X$, and such that for any two
disjoint finite subsets $A,C$ of $X$ the set $\{V\in\V:A\subset V
\wedge V\cap C=\emptyset \}$ is infinite.
\end{lemma}
\begin{proof}
Applying \cite[Theorem~16]{KocSch03} we can find a sequence $\la\U_n:n\in\w\ra$
of clopen $\w$-covers of $X$ such that for any sequence $\la\V_n:n\in\w\ra$, $\V_n\in [\U_n]^{<\w}$,
the union $\bigcup_{n\in\w}\V_n$ fails to be a groupable $\w$-cover of $X$.
Passing to finer $\w$-covers, if necessary, we may additionally assume that
$\U_{n+1}$ is a refinement of $\U_n$ for all $n$, and the projection
of each element of $\U_0$ onto the $0$th coordinate (recall that $\U_0$ is a family of
subsets of $\w^\w$) is finite.

For every $s\in\w^{<\w}$ we shall  denote  the basic open subset  $\{x\in\w^\w:x\uhr |s|=s\}$
of $\w^\w$ by $[s] $. Let $\B=\{B_k:k\in\w\}$ be the family of all finite unions of the sets of the form $[s]$, $s\in\w^{<\w}\setminus
\{\emptyset\}$. It follows from our restrictions on $\U_n$'s that
$\W_{n,k}=\{U\setminus B_k:U\in\U_n\}$ is an $\w$-cover of $X\setminus B_k$ for all $n,k\in\w$
(because no  $U\in\U_n$ contains $X\setminus B_k$).
 Let us decompose $\w$ into countably many disjoint infinite sets $\{I_k:k\in\w\}$
and for every $k\in\w$ consider the sequence $\la \W_{n,k}:n\in I_k \ra$ of clopen $\w$-covers of
$X\setminus B_k$. Since $ X\setminus B_k$ is a clopen subset of $X$,
all of its finite powers have the Menger property, and hence
there exists a sequence $\la \V_{n,k}:n\in I_k \ra$ such that $\V_{n,k}\in [\W_{n,k}]^{<\w}$
and $\bigcup_{n\in I_k}\V_{n,k}$ is an $\w$-cover of $X\setminus B_k$, see \cite{Ark86} or \cite[Theorem~3.9]{COC2}.
Since $\{X\setminus B_k:k\in\w\}$ is an $\w$-cover of $X$, we have that
$\V=\bigcup_{k\in\w,n\in I_k}\V_{n,k}$ is an $\w$-cover of $X$.
Each element of $\V_{n,k}$ is included in some element of $\U_n$, and hence
$\V$ is not groupable.

Finally, let us fix some disjoint finite subsets $A,C\subset X$ and find $k\in\w$
such that $C\subset B_k $ and $A\cap B_k=\emptyset$.  Since $\bigcup_{n\in I_k}\V_{n,k}$
is an $\w$-cover of $X\setminus B_k$, there are infinitely many elements of
$\bigcup_{n\in I_k}\V_{n,k}$ which contain $A$. By the construction, all of them are
disjoint from $C$. This completes our proof.
\end{proof}

\noindent We are now in a position  to complete the proof of Theorem~\ref{main}. \\
In light of Lemma~\ref{easy} it is enough to construct a
 non-meager filter $\F$ and a countable dense $\D\subset\F$ such that
$\F\setminus \D$ has the Menger property. Let $X$ and $\V$ be such as in
Lemma~\ref{easy2}. Let us fix a bijective  enumeration $\{V_n:n\in\w\}$
of $\V$ and for every $m\in\w$ consider the mapping
$f_m:X^m\to \mathcal P(\w) $,
$$f_m\la x_0,\ldots,x_{m-1}\ra=\{n: \{x_0,\ldots,x_{m-1}\}\subset V_n\}.$$
Since $\V$ is an $\w$-cover of $X$ we have $f_m[X^m]\subset [\w]^\w$
for all $m$. Moreover, it is easy to see that
$\mathcal Y=\bigcup_{m\in\w}f_m[X^m]$ is closed under finite intersections of its elements
and hence it generates the filter $\F=\{F\subset\w: Y\subset^* F$ for some $Y\in\mathcal Y\}$.
We claim that $\F$ is as required. Indeed, by  Talagrand and Jalali-Naini's
characterization \cite[Proposition~9.4]{Bla10}
the non-meagerness of  $\F$ is a direct consequence of $\V$  not being groupable.

Let us write $\F$ in the form $\bigcup_{m\in\w}\F_m$, where
$\F_m=\{F\subset\w: Y\subset^* F$ for some $Y\in f_m[X^m]\}$.
Consider the map $g_m:X^m\times \mathcal P(\w)\times\w\to\mathcal P(\w)$, $g_m(a,b,c)=(f_m(a)\setminus c)\cup b$.
It is easy to see that $g_m$ is continuous and $\F_m=g_m[X^m\times \mathcal P(\w)\times\w]$.
Since the Menger property is preserved by products with $\sigma$-compact spaces and
continuous images, we conclude that $\F_m$ has the  Menger property for all $m\in\w$.

Now let us fix any injective sequence $\la x_m:m\in\w\ra$ of elements of $X$
 and set $F_m= f_m\la x_0,\ldots,x_{m-1}\ra$ for all $m>0$. Set also $F_0=\w$.
The sequence  $\la F_m:m\in\w\ra\in\F^\w$ has the property   that
$\{F_m: m \in\w \}\cap\F_k=\{F_m:m \in k+1\}$ for every $k\in\w$.
Indeed, otherwise there exists $\la x'_0,\ldots,x'_{k-1}\ra\in X^k$
such that $\phi_k\la x'_0,\ldots,x'_{k-1}\ra\subset^* F_{k+1}$
(because the sequence  $\la F_m:m\in\w\ra$ is decreasing).
 Since the sequence
$\la x_m:m\in\w\ra$ is injective, there exists $j\leq k$ such that $x_j\not\in \{x'_0,\ldots,x'_{k-1}\}$,
and hence by our choice of $\V$
there exist infinitely many $n\in\w$ such that $\{x'_0,\ldots,x'_{k-1}\}\subset V_n$ and $x_j\not\in V_n$.
However, all these $n$'s are in $\phi_k\la x'_0,\ldots,x'_{k-1}\ra$ but not in
 $f_1(x_j)\supset f_{k+1}\la x_0,\ldots,x_{k}\ra =F_{k+1}$, a contradiction.

Since each $\F_k$ is closed under finite modifications of its elements,
we can conclude that for any sequence $\la F'_m:m\in\w\ra\in\F^\w$,
if $F_m =^* F'_m$ for all $m$, then $\{F'_m: m \in\w \}\cap\F_k=\{F'_m:m \in k+1\}$
for all $k$.

Let  $\{s_m:m\in\w\}$ be an enumeration of $2^{<\w}$ and $F'_m=[F_m\setminus (\mathrm{dom}(s_m))]\cup s_m^{-1}(1)$.
It is clear that $\D=\{F'_m:m\in\w\}$ is dense in $\F$.
It follows from the above that
$$\F\setminus\D = \bigcup_{k\in\w}\F_k\setminus \D=\bigcup_{k\in\w}(\F_k\setminus \{F'_m:m\in k+1\}).$$
Since each $\F_k$ has the  Menger property,  by Corollary~\ref{easy1} we have
  that $\F\setminus\D$ is a countable union of
its subspaces with the Menger property, and hence  itself has the Menger property.
 This completes our proof.
\hfill $\Box$
\medskip

 A space $X$ is called \emph{strongly locally homogeneous}
if it has an open base $\mathcal B$ such that, for each
$U\in\mathcal B$ and points $x,y \in  U,$ there exists a
homeomorphism $h: X \to X$ with $h(x) = y$ and $h\uhr (X\setminus
U)$ equal to the identity. It has been shown in  \cite{AndCurMil82}
that every strongly locally homogeneous Polish space  is CDH. This
result is not anymore true for Baire spaces, even in the realm of
separable metrizable spaces: for every $n\in\w\cup\{\infty\}$
there exists an $n$-dimensional
 Baire space  which is strongly
locally homogeneous but not CDH, see \cite[Remark 4.1]{vMil82}.

However, spaces constructed in \cite{vMil82} are not topological
groups, see Theorem~3.5 there. Thus the filter $\mathcal F$ constructed in the proof of 
Theorem 1 seems to be the first example
of a metrizable separable Baire topological group which is strongly
 locally homogeneous but not CDH. It might be also worth noticing that it cannot be made CDH
by products with metrizable compact spaces.

\begin{proposition}\label{from_proof}
Let $\F$ be the filter constructed in the proof of
Theorem~\ref{main}. Then $\F\times Y$ is not CDH for any metrizable
compact $Y$. 
\end{proposition}
\begin{proof}
We shall use notation from the proof of Theorem~\ref{main}. Let $\la
y_m:m\in\w\ra$ be a sequence of elements of $Y$ such that $\mathcal
D_p:=\{\la F'_m,y_m\ra:m\in\w\}$ is dense in $\F\times Y$. Since
$\F_k\cap\{F'_m:m\in\w\}$ is finite for all $k\in\w$, we have that
$(\F_k\times Y)\cap\mathcal D_p$ is finite for all $k\in\w$. Since
$\F_k$ is Menger, so is $\F_k\times Y$, and hence $(\F_k\times
Y)\setminus\mathcal D_p$ is Menger as well. Therefore $(\F\times
Y)\setminus\mathcal D_p=\bigcup_{k\in\w}(\F_k\times
Y)\setminus\mathcal D_p$ is also Menger. Thus $\F\times Y$ has a
countable dense subset with Menger complement.

On  the other hand, the same argument as in the proof of
Lemma~\ref{easy} implies that $\F\times Y$ has a countable dense
subset whose complements in $\F\times Y$ is not Menger. This
completes our proof.
\end{proof}

We call a filter $\F$ on $\w$ a \emph{$P^+$-filter,} if for every
sequence $\la A_n:n\in\w\ra$ of elements of
$\F^+=\{X\subset\w:\forall F\in\F (X\cap F\neq\emptyset)\}$  there
is a sequence $\la B_n:n\in\w\ra$ such that $B_n\in [A_n]^{<\w}$ and
$\bigcup_{n\in\w}B_n\in\F^+$. Replacing $\F^+$ with $\F$ in the
definition above we get the classical notion of a $P$-filter. Every
filter with the Menger property is a $P^+$-filter. Indeed,  the
Menger property applied to the collection $\{\U_A:A\in\F^+\}$ of
open covers of $\F$, where $\U_A=\{\{X\subset\w:n\in X\}:n\in A\}$,
gives nothing else but the definition of $P^+$-filters. If $\F$ is
an ultrafilter then $\F^+=\F$ and hence $\F$ is a $P$-filter if and
only if it is a $P^+$-filter. Since the non-meager non-CDH filter
constructed in the proof of Theorem~\ref{main} has the Menger
property, we get the following

\begin{corollary}\label{cor}
There exists a non-meager $P^+$-filter which is not CDH.
\end{corollary}

 Corollary~\ref{cor} follows directly from  Theorem~\ref{main},
and hence its proof does not require anything beyond ZFC.
Corollary~\ref{cor} implies that one of the main results of
\cite{HruGot??} which states that non-meager $P$-filters are  CDH is
sharp in the sense that it cannot be extended to $P^+$-filters, even
under  additional set-theoretic assumptions.

Following \cite{BanZdo} we call filters $\F_0$ and $\F_1$ on $\w$ \emph{coherent},
if there exists a monotone surjection $\psi:\w\to\w$ such that $\psi[\F_0]=\psi[\F_1]$.
It is easy to see that the coherence is an equivalence relation.
It has been shown in \cite{BlaLaf89} that in the  model
constructed by Miller in \cite{Mil84}, any two non-meager filters are coherent
and  there exists a $P$-point, i.e., an ultrafilter which is a $P$-filter.
 Together with Theorem~\ref{main} this proves the following

\begin{corollary} \label{cor1}
In the Miller model the collection of all CDH filters is not closed under
the coherence relation.
\end{corollary}

We do not know whether Corollary~\ref{cor1} is true in ZFC.
Let us note that a ZFC proof of it would require at least a  construction
of a CDH filter without any additional set-theoretic assumption,
which seems to be quite a difficult task.

We recall that $\dd$ is, by  definition, the minimal cardinality
of a cover of $\w^\w$ by its compact subspaces, and $\mathfrak u$ is
the minimal cardinality of a base of an ultrafilter on $\w$. We
refer the reader to \cite{Bla10} for more information on
$\mathfrak u$, $\dd$, and other  cardinals characteristics of the
continuum.
\medskip

\noindent\textbf{Remark.} \ In the proof of Theorem~\ref{main} we
had
 to be rather careful with the choice of $\V$. This is because  the same argument would not work
 if we started with  $\V$ which satisfies all requirements of Lemma~\ref{easy2} except for the last one,
i.e., does not allow to sufficiently distinguish  between disjoint finite subsets of
$X$. Indeed, assume $\mathfrak u<\dd$, which holds, e.g.,
 in the aforementioned Miller model,
 and let $\U$ be an ultrafilter with $\mathfrak u$-many generators.
Then $\U$ fails to have the Hurewicz property by \cite[Theorem~4.3]{COC2} being a non-meager subset of
$[\w]^\w$. On the other hand, by \cite[Theorem~4.4]{COC2} all finite powers of $\U$ have the Menger property because
$\U$ is a union of fewer than $\dd$ compact spaces. Now set $\V=\{V_n:n\in\w\}$,
where $V_n=\{U\in\U:n\in U\}$. It is easy to check that $\V$ is an $\w$-cover of $\U$ which fails to be
groupable. Letting $X=\U$ and defining $\phi_m$'s and $\F$ in the same way as in the proof of
Theorem~\ref{main}, one can easily check that $\F=\U$. However,
$\U$ is a $P$-point, see, e.g.,  \cite[Theorem~9.25]{Bla10}. Therefore $\U$ is CDH \cite{HruGot??}
and hence by Lemma~\ref{easy} it is impossible to select a countable dense $\D\subset \U$ such that $\U\setminus\D$
has the Menger property.
\hfill $\Box$
\medskip

\noindent\textbf{Acknowledgment.} The second author  would like to
thank M. Hru\v{s}\'ak and S. Todor\v{c}evi\'c for  useful
discussions during the Winter School in abstract Analysis (Set
Theory and Topology Section) held in Hejnice in January 2013. Also,
we are grateful to A. Medini for bringing the paper \cite{vMil82} to
our attention and for detecting some inaccuracies in the previous versions.

The first author acknowledges the support
of the ARRS  grant  P1-0292-0101.
The second author would
like to  thank   the Austrian Academy of Sciences (APART Program) as
well as the Austrian Science Fund FWF (Grants M 1244-N13 and I 1209-N2)
   for generous support for this research. Parts of the work reported here were carried out during the visit of
the third named author at the  Kurt G\"odel Research Center in
November 2012. This visit was supported by the above-mentioned FWF
grant. The third  author thanks the second  one for his kind
hospitality. The third author would also like to acknowledge the
support of the NSFC grant \#11271272.

\end{document}